\documentclass{ws-ijnt}
\usepackage[section]{placeins}

\begin{document}
	\markboth{Hongnan Chen , Fenglin Huang, Sihui Zhang}{On the length of an arithmetic progression}
	
	
	\title{On the length of an arithmetic progression of the form ${3^x+2^y}$
	}
	
	\author{Hongnan Chen}
	
	\address{College of Science, University of Shanghai for Science and Technology, Shanghai\\
		Shanghai, 200093, China\\
		\email{chn336699@126.com} }
	
	\author{Fenglin Huang}
	
	\address{School of Mathematical and Sciences, Peking University, Beijing\\
		Beijing, 100091, China\\
		2000010755@stu.pku.edu.cn}
	
	\author{Sihui Zhang\footnote{
			corresponding author}}
	
	\address{College of Science, University of Shanghai for Science and Technology, Shanghai\\
		Shanghai, 200093, China\\
		071018005@fudan.edu.cn}

	\maketitle

	\begin{abstract}
		The conclusion that the length of an arithmetic progression of the form ${3^x+2^y}$ is at most six is proved.
	\end{abstract}
	
	

	\section{Introduction} \label{section1}
	
	The problem of a positive integers are represented as ${3^x+2^y}$ has been an interestng topic in the theoretical study of mathematics. Iannucci, D.E proved in [1] that there are exactly five positive integers that
	can be written in more than one way as the sum of a nonnegative
	power of 2 and a nonnegative power of 3. This proposition is given below.
	
	\begin{proposition} {\rm (cf. [1])}\ \ There are exactly five positive integers that can be written in multiple ways as the sum of non-negative powers of 2 and non-negative powers of 3, and there are only five integers that can be expressed in two ways as  ${3^a+2^b\left( a,b\in N\right) }$: 5, 11, 17, 35, 259. The five elements listed are$$5=2^2+3^0=2^1+3^1,\ 11=2^3+3^1=2^1+3^2,\ \ 17=2^4+3^0=2^3+3^2,$$ $$ 35=2^5+3^1=2^3+3^3,\ \ 259=2^8+3^1=2^4+3^5.$$ 
		\label{prop 1.1}
	\end{proposition}

	Pillai states that for each fixed integer $c\ge1$, the Diophantine equation
	$$a^x-b^y=c, min{\left\lbrace x,y\right\rbrace }\ge2$$
	has only a finite number of positive solutions ${a,b,x,y}$. This problem is still open, it is still attracting a large number of researchers to study this issue.
	
	M. Ddamulira in [2] found that all integers $c$ having at least two representations as a difference between a Fibonacci number and
	a power of 3. Later, M. Ddamulira et al in [3] porved that all integers $c$ with at least two representations as a difference between a k-generalized Fibonacci number and a power of 3. In this paper, we prove the following conclusion.
	\begin{theorem}\label{thm 1.1} The length of an arithmetic progression of the form ${3^x+2^y}$ for nonnegative integers $x$ and $y$ is at most six.
		
	\end{theorem}
	
	\section{Proof of Theorem \ref{thm 1.1}}
	\label{section2}
	
	The example of an arithmetic progression of the form ${3^x+2^y}$ of length six for non-negative integers $x$ and $y$ is given below. The six elements listed are $$ 3=2^1+3^0,\ \ 5=2^2+3^0=2^1+3^1,\ \ 7=2^2+3^1,$$ $$ 9=2^0+3^2,\ \ 11=2^1+3^2=2^3+3^1,\ \ 13=2^2+3^2. $$ 
	
	In the following, we will prove that an arithmetic progression of the form ${3^x+2^y}$ of length seven for non-negative integers $x$ and $y$ does not exist.
	
	If there is an A.P. with a length of 7 in  $\left \{ 3^m+2^n|m,n \in \mathbb{N} \right \} $,  let it to	
	$${a_1=a+6d}, {a_2=a+5d}, \dots , {a_6=a+d}, {a_7=a}.$$Let $n=\left [ log_{2}a_1  \right ] $,  $m=\left [ log_{3}a_1  \right ] $.
	
	When $m \le 8$, we give the following program to enumerate, see appendix for program details. Therefore, we consider the case $m \ge 9 $.

	\subsection{The properties of $d$}
	\label{section 2.1}
	
	In this section, we will prove the three properties that $d\geq500$, $d$ is even, and $3\mid d$. These three lemmas are given below.
	
	\begin{lemma}   $d\geq500$.
		\label{lem2.1}
	\end{lemma}
	\begin{proof} If $d<500$, since $m \ge 9 $, $${a_1} \ge {3^m} \ge {3^9} =19683.$$
		Thus $${a_7}={{a_1}-6d} \ge {{a_1}-3000} \ge {\frac{5}{6}{a_1}} \ \ (with ~\frac{{a_1}}{6} >3000).$$
		Thus ${a_1},{a_2},{a_3},{a_4},{a_5},{a_6},{a_7}$ is ${3^m+2^y}$ or ${2^n+3^x}$ type, this means that four of the numbers are of the same type.
		
		For a prime $p$ and a natural number $n$, we write $p||n$ if $p \mid n$ but $p^2 \nmid n$. We 
		denote the p-valuation of $n$ by ${\nu_p}(n)$: i.e., ${\nu_p}(n) = k$ if ${p^k}||n.$
		
		(1) If we have four ${3^m+2^y}$, let ${3^m+2^{p_1}}$, ${3^m+2^{p_2}}$, ${3^m+2^{p_3}}$, ${3^m+2^{p_4}}$ with $ {p_1} >{p_2} >{p_3}>{p_4} $. Since $${\nu_2}({2^{p_1}-2^{p_2}}) \ge 2+{\nu_2}({2^{p_3}-2^{p_4}}) \ge {2+{\nu_2}(d)},$$ 
		and $${\nu_2}({2^{p_2}-2^{p_3}}) \ge 1+{\nu_2}({2^{p_3}-2^{p_4}}) \ge {1+{\nu_2}(d)}.$$ 
		Thus $$ \frac{{2^{p_1}-2^{p_2}}}{d} \ge 4,  \frac{{2^{p_2}-2^{p_3}}}{d} \ge 2.$$ 
		Therefore $$(3^m+2^{p_1})-(3^m+2^{p_4}) \ge 7d.$$ This is a contradiction.
		
		(2) If we have four ${3^x+2^n}$, let ${3^{q_1}+2^n}$, ${3^{q_2}+2^n}$, ${3^{q_3}+2^n}$, ${3^{q_4}+2^n}$ with $ {q_1} >{q_2} >{q_3} >{q_4} $. Since $${\nu_3}({3^{q_1}-3^{q_2}}) \ge 2+{\nu_2}({3^{q_3}-3^{q_4}}) \ge {2+{\nu_3}(d)},$$ 
		thus $$ \frac{{3^{q_1}-3^{q_2}}}{d} \ge 9,$$ therefore $$(3^{q_1}+2^n)-(3^{q_2}+2^n) \ge 9d.$$ This is a contradiction.
	\end{proof}
	
	\begin{lemma} $2 \mid d$.
		\label{lem2.2}
	\end{lemma}
	
	\begin{proof}  If d is odd, then there must exist three even numbers and all are denoted as $3^x+2^y$. Since $2\mid 3^x+2^y$, $3^x$ is odd, thus $y=0$. Therefore, all three numbers are expressed in the form $3^x+1$. Thus, let ${1+3^{S_1}}, {1+3^{S_2}}, {1+3^{S_3}}$ where $ {S_1} > {S_2} >{S_3} $.
		
		Since ${S_1}> {S_2}$, thus $${1+3^{S_1}}-({1+3^{S_2}})={3^{S_1}-3^{S_2}} >3^{S_2}>{3^{S_2}-3^{S_3}}.$$ This is a contradiction.
	\end{proof}
	
	\begin{remark} If there exists ${y_i}=0$, then there must exist three even numbers. The situation is the same as the first proof, this is a contradiction.
		\label{rem2.1}
	\end{remark}

	\begin{lemma} $3 \mid d$.
		\label{lem2.3}
	\end{lemma}
	\begin{proof}
		If $3\nmid d$, thus we consider that ${a_i}=3^{x_i}+2^{y_i}$, ${a_{i+1}}=3^{x_{i+1}}+2^{y_{i+1}}$, ${a_{i+2}}=3^{x_{i+2}}+2^{y_{i+2}}$. 
		
		Thus $d={a_i}-{a_{i+1}}={a_{i+1}}-{a_{i+2}}$ not a multiple of 3, and $2d={a_i}-{a_{i+2}}$ not a multiple of 3.
		
		We consider the parity of ${y_i}$, ${y_{i+1}}$, ${y_{i+2}}$, where two numbers must have the same parity, thus their difference in powers of 2 must be a multiple of 3, it may be useful to set the subscripts to j and k, thus $3\mid(2^{y_j}-2^{y_k})$.
		
		However $3\nmid d$, thus $${a_j}-{a_k}=(3^{x_j}+2^{y_j})-(3^{x_k}+2^{y_k}) $$ not a multiple of 3, thus $3^{x_j}-3^{x_k} $ not a multiple of 3.
		
		This shows that ${x_j}=0$ or ${x_k}=0$, thereby at least one of the three numbers  ${x_i},{x_{i+1}},{x_{i+2}}$ is equal to 0, and $i=1,2,3,4,5$. Thus at least one of the three numbers  ${x_1},{x_2},{x_3}$ is equal to 0, and at least one of the three numbers  ${x_4},{x_5},{x_6}$ is equal to 0.     
		
		Thus at least two of the six numbers ${x_1},{x_2},{x_3},{x_4},{x_5},{x_6}$ are equal to 0, let $i,j$ such that ${x_i}={x_j}=0$ and $1 \le i < j \le 6$, $(i,j)$ is the smallest set of them. Then there must be \begin{equation}\label{2.1}
			j-i \le 3
		\end{equation}
		otherwise $j-i >3$, there must be $0$ in the three numbers ${x_{i+1}},{x_{i+2}},{x_{i+3}}$, which contradicts the minimality of $i,j$.
		
		Let ${a_i}=2^{y_i}+1, {a_j}=2^{y_j}+1$,
		since $${a_j} \ge {a_6} \ge {{a_7}+d}, {a_7}\ge {3^0}+{2^0}=2,$$
		thus ${a_j} \ge d+2$. Meanwhile from \eqref{2.1} we have ${a_i}-{a_j}=(j-i)d \le 3d$, thus $$2^{y_i}-2^{y_j} \le 3d \le 3({a_j}-2) =3\cdot(2^{y_j}+1-2)=3\cdot2^{y_j}-3,$$
		thus ${2^{y_i}} \le {4\cdot2^{y_j}-3}$, therefore ${y_i}< {y_j}+2$. Since ${y_i} >{y_j}$, we have ${y_i}={y_j}+1$, thus ${a_i}-{a_j}=2^{y_j}=(j-i)d $.
		Since $1 \le j-i \le 3 $, thus $$d=2^{y_j}~or~d=2^{{y_j}-1},$$ thereby 
		\begin{equation}\label{2.2}
			d \ge 2^{{y_j}-1},~j-i=1~or~2
		\end{equation}
		thus $j-i \ne 3$, therefore $j-i \le 2$. Since $i \le 3$, thus $j \le 5$. This implies that ${a_j}-2d \ge {a_7}$. From \eqref{2.2} we have  $${a_j}-2d \le (2^{y_j}+1)-2\cdot2^{{y_j}-1} =1,$$
		thus ${a_7} \le 1$. This is a contradiction.
	\end{proof}

	\subsection{The type of ${a_1}, \dots, {a_4}$}
	\label{section 2.2}
	
	In this section we discuss all the possible types of ${a_1}, \dots, {a_4}$. We show that there are four possible types for ${a_4}$, two possible types for ${a_1}$ and ${a_2}$, and three possible types for ${a_3}$.
	
	\begin{proposition} \label{prop2.1}
		${a_4}=a+3d$ can only be written as one of the four numbers  $${3^m+2^{y_4}},{3^{m-1}+2^{y_4}},{2^{n-1}+3^{x_4}},{2^n+3^{x_4}}.$$
		
	\end{proposition}
\begin{proof}
	 $$3^m \le {a_1}=a+6d < 2^{n+1}, 2^n \le  {a_1}=a+6d <3^{m+1}.$$
	Thus\begin{equation*}
		\begin{aligned}
			3^{m-2}+2^{n-2}&=\frac{1}{9}\cdot3^m+\frac{1}{4} \cdot2^n < \frac{1}{9} \cdot2^{n+1}+\frac{1}{4} \cdot2^n\\&=\frac{17}{36} \cdot2^n \le \frac{17}{36} (a+6d) <a+3d.
		\end{aligned}
	\end{equation*}
	Thus $a+3d=3^{x_4}+2^{y_4}$ where ${x_4} \ge m-1$ or ${y_4} \ge n-1$.
	Thus $a+3d$ has four possibilities: $${3^m+2^{y_4}},{3^{m-1}+2^{y_4}},{2^{n-1}+3^{x_4}},{2^n+3^{x_4}}.$$
\end{proof}

	\begin{proposition}\label{prop2.2}
		${a_1}=a+6d, {a_2}=a+5d$  must  be expressed as   $${3^m+2^y}~or~{2^n+3^x}.$$
		
	\end{proposition}
	\begin{proof}
		If ${x_2} \ne m$ and ${y_2} \ne n$. Since ${x_2} \le m, {y_2} \le n$, thus $${x_2} \le m-1, {y_2} \le n-1.$$
		Therefore\begin{equation*}
			\begin{aligned}
				a+5d&=3^{x_2}+2^{y_2} \le 3^{m-1}+2^{n-1} =\frac{1}{3}\cdot3^m +\frac{1}{2}\cdot2^n \\ &\le \frac{1}{3}(a+6d)+\frac{1}{2}(a+6d) = \frac{5}{6}a+5d.
			\end{aligned}
		\end{equation*}
		The above inequality holds because we have $2^n \le a+6d, 3^m \le a+6d$. This is a contradiction. 
		
		Thus ${x_2}=m$ or ${y_2}=n$, thereby
		${a_2}=a+5d={3^m+2^{y_2}}$ or ${2^n+3^{x_2}}$.
		
		By the same token we can obtain ${x_1}=m$ or ${y_1}=n$, thereby
		${a_1}=a+6d={3^m+2^{y_1}}$ or ${2^n+3^{x_1}}$.
	\end{proof} 
	
	\begin{proposition}	\label{prop2.3}
		${a_3}=a+4d$  can only be written as one of the three numbers   $${3^m+2^y} , {2^n+3^x},{3^{m-1}+2^{n-1}}.$$
	\end{proposition}
	\begin{proof}
		we consider that 
		\begin{equation*}
			\begin{aligned}
				2^{n-1}+3^{m-2}&= \frac{1}{2}\cdot2^n+ \frac{1}{9}\cdot3^m  \le \frac{1}{2}\cdot(a+6d)+ \frac{1}{9}\cdot(a+6d)\\
				&=\frac{11}{18}\cdot a+ \frac{11}{3}\cdot d <\frac{11}{18}\cdot a +4d <a+4d,
			\end{aligned}
		\end{equation*}
		
		\begin{equation*}
			\begin{aligned}
				2^{n-2}+3^{m-1}&= \frac{1}{4}\cdot2^n+ \frac{1}{3}\cdot3^m  \le \frac{1}{4}\cdot(a+6d)+ \frac{1}{3}\cdot(a+6d)\\&=\frac{7}{12}\cdot a+ \frac{7}{2}\cdot d <\frac{7}{12}\cdot a+4d <a+4d.
			\end{aligned}
		\end{equation*}
		
		Thus, $a+4d\ge 2^{n-1}+3^{m-2}$, and $a+4d\ge 2^{n-2}+3^{m-1}$, then ${a_3}=a+4d$  can only be written as one of the three numbers   $${3^m+2^y} , {2^n+3^x},{3^{m-1}+2^{n-1}}.$$
	\end{proof}
	
	\subsection{${a_1}, \dots, {a_4}$ Specific discussion}
	\label{section 2.3}
	
	In this section, for the four numbers ${3^m+2^{y_i}}$, ${3^{x_i}+2^n}$, ${3^{m-1}+2^{y_i}}$ and ${3^{x_i}+2^{n-1}}$, we will call them $3$-dominated, $2$-dominated, $3$-weak dominated and $2$-weak dominated, respectively.
	
	Below we first prove that there cannot be three $3$-dominated or three $2$-dominated or two $3$-dominated and two $2$-dominated in ${a_1}, \dots, {a_4}$. Then we prove that
	${a_3}\neq 3^{m-1}+2^{n-1}$, ${a_4} \neq {2^{n-1}+3^{x_4}}$ and ${a_4} \neq {3^{m-1}+2^{y_4}}$. In summary, it can be shown that all cases of ${a_1}, \dots, {a_4}$ are contradictory.

	\begin{prop} \label{prop2.4}
		There can be no three numbers in ${a_1}, \dots, {a_4}$ expressed as $3$-dominated.
	\end{prop}
	\begin{proof}
		If there are three $3$-dominated in ${a_1}, \dots, {a_4}$, denoted  ${3^m+2^p}, {3^m+2^q}$, ${3^m+2^r}$ where $p> q > r$. Then,  $$d\mid({3^m+2^p})-({3^m+2^q}) \le 2d,$$ $$d\mid({3^m+2^q})-({3^m+2^r}) \geq d,$$ thus $${2^p-2^q }> {2^q-2^r}\ge d ,$$ $${2^p-2^q }> {2^q-2^r} = ({3^m+2^q})-({3^m+2^r}) \equiv 0~(\mod~\mathrm{d}).$$ Therefore$$d=({3^m+2^q})-({3^m+2^r})={2^q-2^r},$$
		$$2d=({3^m+2^p})-({3^m+2^q} )={2^p-2^q}={2^{q+1}-2^{r+1}}.$$From this, we can see $$3^m+2^p=a_1, 3^m+2^q=a_3, 3^m+2^r=a_4.$$ Thus $$p=q+1=r+2, d={2^q-2^r}={2^{r+1}-2^r}=2^r.$$
		
		If ${a_5}=3^{x_5}+2^{y_5}={a_4}-d={({3^m+2^r})-2^r}=3^m \equiv 0\left(\mod~\mathrm{3}\right) $, thus ${x_5}=0$. Thus $$3^{x_5}+2^{y_5}=3^0+2^{y_5}=1+2^{y_5}=3^m,$$
		thus \begin{equation}\label{2.3}
			2^{y_5}+3^1=3^m+2^1
		\end{equation}

		According to proposition \ref{prop 1.1}, \eqref{2.3} is no solution if $m \ge 6$. This is a contradiction.
		
	\end{proof}

	\begin{proposition} There can be no three numbers in ${a_1}, \dots, {a_4}$ expressed as $2$-dominated.
		\label{prop 2.5}
	\end{proposition}
	\begin{proof}
		If there are three $2$-dominated in ${a_1}, \dots, {a_4}$, denoted ${2^n+3^p}$, ${2^n+3^q}$, ${2^n+3^r}$ where $p> q > r$,
		thus $$3d \ge (2^n+3^p)- (2^n+3^r) =3^p-3^r \ge 3^{q+1}-3^r >3(3^q-3^r) \ge 3d.$$ This is a contradiction.
	\end{proof} 
	
	\begin{proposition} There cannot be represented as exactly two $2$-dominated and two $3$-dominated. 
		\label{prop 2.6} 
	\end{proposition}
	\begin{proof}
		If ${a_1}, \dots, {a_4}$ is the arrangement of ${2^n+3^x}, {2^n+3^y},{3^m+2^z}$ and ${3^m+2^w}$, without loss of generality, we assume $x>y$ and $z>w$.
		
		(1) If ${a_1},{a_2}$ are the same type, then ${a_3},{a_4}$ are the same type, thus $$d= {3^x-3^y}={2^z-2^w},$$ thereby $d< 500.$ This is a contradiction.
		
		(2) If ${a_1},{a_3}$ are the same type, then ${a_2},{a_4}$ are the same type. thus $$2d= {3^x-3^y}={2^z-2^w},$$ thereby $d< 500.$ This is a contradiction.
		
		(3) If ${a_1}={2^n+3^x},{a_4}={2^n+3^y}$. Thus ${3d={a_1}-{a_4}=3^x-3^y}$ where $y>1$, and $${a_2}-{a_3}=d=({3^m+2^z})- ({3^m+2^w})={2^z}-{2^w}=3^{x-1}-3^{y-1},$$ thereby  $d< 500$. This is a contradiction.
		
		(4) If ${a_1}={3^m+2^z},{a_4}={3^m+2^w}$. Thus ${3d={a_1}-{a_4}=2^z-2^w}$, and ${d={a_2}-{a_3}=3^x-3^y} $. Thus $$3d=2^z-2^w=3^{x+1}-3^{y+1},$$ thereby $d< 500$. This is a contradiction.
		
		Combining Proposition \ref{prop2.4}, \ref{prop 2.5} and \ref{prop 2.6}, ${a_1}, \dots, {a_4}$ can't be all $3$-dominated or $2$-dominated.
	\end{proof}

	\begin{proposition}
		${a_3}\neq 3^{m-1}+2^{n-1}$.
		\label{prop2.7}
	\end{proposition}
	\begin{proof}
		According to a lemma \ref{lem2.1}, \ref{lem2.2}, \ref{lem2.3}, we obtain that: ${6\mid d}$, $d\ge 500$, 
		$y_i$ is a positive integer for all $a+(7-i)d=3^{x_i}+2^{y_i}$.

		It has already been demonstrated that
		${a_1},{a_2}$ is $3$-dominated or $2$-dominated type, ${a_3}$ can only be $3$-dominated, $2$-dominated, ${3^{m-1}+2^{n-1}}$ and ${a_4}$ can only be $3$-dominated, $2$-dominated, $3$-weak dominated, $2$-weak dominated.
		
		If ${a_3}=3^{m-1}+2^{n-1}$, the possibility of ${a_1}$ is discussed below.
		
		(1) If ${a_1}=3^m+2^{y_1}$, we have $$4\mid 2d={a_1}-{a_3}=2\cdot3^{m-1}+2^{y_1}-2^{n-1}.$$ Thus ${y_1}=1, {a_1}=3^m+2$. Applying Proposition \ref{prop2.2} and the fact that $6\mid d$, we obtain that ${a_2}=3^{x_2}+2^n$, thus $$3\mid d={a_2}-{a_3}=(3^{x_2}+2^n)-(3^{m-1}+2^{n-1}),$$ therefore ${x_2}=0$, ${a_2}=1+2^n$, thus ${a_1}+{a_3}=2\cdot{a_2}$, thereby $3^m+2+3^{m-1}+2^{n-1}=2\cdot2^{n+1}$, therefore $4\cdot3^{m-1}=3\cdot2^{n-1}$, thus $$3^{m-2}=2^{n-3}.$$ Combing with the fact that $m\geq 6$ yields a contradiction.
		
		(2) If ${a_1}=3^{x_1}+2^n$, thus $$3\mid 2d={a_1}-{a_3}=3^{x_1}+2^n-2^{n-1}-3^{m-1},$$ thereby ${x_1}=0$, thus ${a_1}=2^n+1$. Applying Proposition \ref{prop2.2} shows that $${a_2}=3^m+2^{y_2},$$ $$4\mid 2d={a_1}-{a_3}=2^n+1-2^{n-1}-3^{m-1}=2^{n-1}-(3^{m-1}-1).$$
		Thus $4\mid 3^{m-1}-1$, $2\mid m-1$, $8\mid 3^{m-1}-1$, therefore $8\mid 2d$, thus $$4\mid d={a_2}-{a_3}=3^m+2^{y_2}-2^{n-1}-3^{m-1}=2\cdot3^{m-1}+2^{y_2}-2^{n-1}.$$
		Thus ${y_2}=1$, thereby ${a_2}=3^m+2$, and ${a_1}+{a_3}=2\cdot{a_2}$, thus $$(2^n+1)+(3^{m-1}+2^{n-1})=2\cdot3^m+4.$$Thus $5\cdot3^{m-1}=3\cdot2^{n-1}-3$,
		thus $5\cdot3^{m-2}=2^{n-1}-1$. Since $m \ge 6$, thus $9\mid 2^{n-1}-1$, thereby     $6\mid n-1$, thereby $63\mid 2^{n-1}-1$. 
		
		Finally, it can be introduced that $7\mid 2^{n-1}-1$. But $7\nmid 5\cdot3^{m-2},$ this is a contradiction.
		
		In conclusion, ${a_3} \ne 3^{m-1}+2^{n-1}$, thus ${a_1},{a_2},{a_3}$ is $3$-dominated or $2$-dominated, therefore ${a_4}$ is not this type. Applying Proposition \ref{prop2.1}, it suffices to consider the following two scenarios: ${a_4}={3^{m-1}+2^{y_4}}$ or ${2^{n-1}+3^{x_4}}$.
	\end{proof}

	\begin{proposition}
		${a_4} \neq {2^{n-1}+3^{x_4}}$.
		\label{prop2.8}
	\end{proposition}
	\begin{proof}
		When ${a_4}=3^{x_4}+2^{n-1},$ then $$2^n+2\cdot3^{x_4}=2\cdot{a_4}={a_1}+{a_7} > 2^n+2.$$
		where the inequality comes from the fact that Lemma \ref{lem2.2}, ${a_1}, \dots, {a_7}$ is odd and ${a_1} \ge 2^n$. Thus ${a_1} \ge 2^n+1$, ${a_7} \ge 2^1+1$, thus ${x_4} \ge 1$.
		
		Since ${a_1},{a_2},{a_3},{a_4}$ not three $3$-dominated, there must be ${2^n+3^x}$ in ${a_1},{a_2},{a_3}$, take the one with the smallest $x$. Since $$3\mid d\mid 2^n+3^x-(3^{x_4}+2^{n-1}),$$ thereby $x=0~({x_4} \ge 0)$. Thus there exists $1 \le i \le 3$, ${a_i}=2^n+1$.
		
		For $1 \le j \le 3$, $j \ne i$, $$3\mid d\mid({a_j}-{a_i})=3^{x_j}+(2^{y_j}-2^n)-1.$$
		If ${x_j}=0$, thus ${y_j}=n$, ${a_i}={a_j}$, this is a contradiction. Thus ${x_j}\ge 1$, $3\mid 2^{y_j}-2^n$. Therefore, ${x_j}=m$, thus $${a_j}=3^m+2^{y_j}.$$
		
		We consider the three cases of $${a_i}=2^n+1, 1 \le i \le 3.$$
		
		{\bf Case 1:} ${a_3}=2^n+1$. Then
		$${a_1}={3^m+2^{y_1}}, {a_2}=3^m+2^{y_2}, {a_3}={1+2^n},{a_4}={3^{x_4}+2^{n-1}}.$$
		
		Since ${3^m \le a+6d \le 2^{n+1}}$, thus
		\begin{equation*}
			\begin{aligned}
				{a+3d}={a_4}&={3^{x_4}+2^{n-1}} \le{3^{m-2}+2^{n-1}} \\ &\le  \frac{1}{2}\cdot{2^n}+ \frac{1}{9}\cdot{3^m} \le \frac{1}{2}\cdot{2^n}+ \frac{1}{9}\cdot{2^{n+1}}\\&=\frac{13}{18}\cdot{2^n} < \frac{3}{4}\cdot{2^n} < \frac{3}{4}(2^n+1) =  \frac{3}{4}(a+4d).
			\end{aligned}
		\end{equation*}
		This is a contradiction.
		
		{\bf Case 2:}  ${a_2}=2^n+1$. Then
		$${a_1}={3^m+2^{y_1}}, {a_2}={1+2^n}, {a_3}={3^m+2^{y_3}},{a_4}={3^{x_4}+2^{n-1}}.$$
		
		Since
		${4(a+3d) >3(a+4d)}$, thereby $4(3^{m-2}+2^{n-1}) > 3\cdot3^m$, thereby $${2^{n+1} >23\cdot3^{m-2}}.$$
		Meanwhile ${5(a+3d) >3(a+4d)}$, thereby $5(3^{m-2}+2^{n-1}) > 3\cdot 2^n $, thereby $$ {5\cdot 3^{m-2} >2^{n-1}}.$$
		From above we obtain $$2^{n-1} >\frac{23}{4}\cdot 3^{m-2}>{5\cdot 3^{m-2} >2^{n-1}},$$
		this is a contradiction.
		
		{\bf Case 3:}  ${a_1}=2^n+1$. Then
		$${a_1}={1+2^n}, {a_2}={3^m+2^{y_2}}, {a_3}={3^m+2^{y_3}},{a_4}={3^{x_4}+2^{n-1}}.$$
		
		(1) If ${\nu_2}(d)=1.$
		
		Since ${y_2} > {y_3} \ge 1$, thereby ${y_2} \ge 2$; Since ${2\mid d}$, thereby $$4\mid 2d=(3^m+2^{y_2})-(3^{x_4}+2^{n-1}) \equiv 3^m-3^{x_4}~(\mod~4).$$ Therefore, ${x_4}$ has the same parity as $m$.
		
		Since ${\nu_2}(d)=1$, we have 
		$$4 \nmid {3^m+2^{y_3}-3^{x_4}-2^{n-1}},$$thereby ${y_3}=1$.
		But ${\nu_2}(d)=1$, thus ${\nu_2}(2d)=2.$ Thereby $${{\nu_2}(3^m+2^{y_2}-3^{x_4}-2^{n-1}) =2},$$ thus ${y_2}=2$, which implies that $d=2$. This is a contradiction.
		
		(2) If ${\nu_2}(d) \ge 2$.
		
		we hace $$4\mid 3d =({2^n}+1)(2^{n-1}+3^{x_4}).$$ Therefore, 
		${x_4}$ is an even number, thus $$3d \equiv {3^{x_4}-1} \equiv 0~(\mod~8),$$ thereby  $8\mid d$.
		
		If $y_2\leq n-5$, then combining with the fact that ${5(a+6d)} < {6(a+5d)}$ and ${3(a+4d)} < {4(a+3d)}$, it yields that ${77\cdot2^{n-4} <6\cdot3^m}$ and ${2^{n+1} >23\cdot3^{m-2}}$. This is a contradiction. Thus ${y_2} \ge n-4$. 
		
		Since ${m>6}$, we get that ${n>7}$ and   ${{y_2} \ge 4}$.
		Thus $${16\mid 2d} ={2^{y_2}+3^m-(2^{n-1}+3^{x_4}) \equiv {3^m-3^{x_4}}(\mod~16)},$$
		thereby ${x_4} \equiv m~(\mod~4)$, thus ${x_4} \le m-4$.
		
		Since $${3(a+4d) < 4(a+3d)},$$ thus
		\begin{equation}\label{2.4}
			{2^{n+1} >239\cdot3^{m-4}};
		\end{equation}
		If ${y_2} \le n-3$,
		since  $${5(a+6d) < 6(a+5d)},$$ thus$$5(2^n+1) <6(3^m+ 2^{n+3}),$$ thereby
		\begin{equation}\label{2.5}
			7\cdot2^{n-2} <6\cdot3^m .
		\end{equation}
		Combining \eqref{2.4} and \eqref{2.5} shows that this is contradictory. 
		
		Therefore ${y_2} >n-3$, and $$3\mid d =2^n+1-(3^m+2^{y_2}),$$ therefore, ${y_2}$ and $n$  have different parity, thus ${y_2}=n-1$.
		
		Since $$5(a+3d)>3(a+5d),$$
		thus $$5(3^{x_4}+2^{n-1}) > 3(3^m+2^{n-1}),$$ thereby
		\begin{equation}\label{2.6}
			{2^n} > {3^{m+1}-5\cdot3^{x_4}} \ge {3^{m+1}-5\cdot3^{m-4}} = 238\cdot3^{m-4}.
		\end{equation}
		Combining \eqref{2.4} and \eqref{2.6} shows that this is contradictory. 
	\end{proof} 
	
	\begin{proposition}
		${a_4} \neq {3^{m-1}+2^{y_4}}$.
		\label{prop2.9}
	\end{proposition}
	
	\begin{proof}
		If ${a_4}=3^{m-1}+2^{y_4}$. Since $${2\cdot3^{m-1}+2^{{y_4}+1}=2{a_4}={a_7}+{a_1} > 3^m},$$and $m>11$,
		thus ${y_4}> 4 $.
		
		Since there must be a number in ${a_1},{a_2},{a_3}$ that is $3$-dominated, suppose $3^m+2^y$ is smallest. Since $d$ is a multiple of 2, therefore $$2\mid 3^m+2^y-(3^{m-1}+2^{y_4}),$$ Thus $y \ge 1$.
		
		We are going to prove $y=1$ by contradiction.
		
		If $y \ge 2$, thus $${3^m+2^y}-(3^{m-1}+2^{y_4})=2\cdot3^{m-1}+2^y-2^{y_4} \equiv 2~(\mod~4).$$
		Thus $d$ is not a multiple of 4, thereby  $d \equiv 2~(\mod~4)$. Thus $$2d ={a_2}-{a_4}=3^{x_2}+2^{y_2}-(3^{m-1}+2^{y_4}) \equiv 4~(\mod ~8).$$

		This is discussed below in two cases.
		
		{\bf Case 1:}\ \ When ${x_2} \equiv m-1~(\mod~2)$, thus ${x_2} \ne m$, thus ${y_2}=n>3$, therefore $$2d \equiv 3^{x_2}-3^{m-1} \equiv 0~(\mod~8).$$ This is a contradiction. 
		
		{\bf Case 2:}\ \ When ${x_2} \not\equiv m-1~(\mod~2)$, thus ${3^{x_2}-3^{m-1}}$ not a multiple of 4, but $4\mid 2^{y_4}$,
		thus $2^{y_2}$ not a multiple of 4, therefore ${y_2}=1<y$. This contradicts the minimality of $y$.
		
		Combining case 1 and case 2 shows that $y=1$. Thus there exists $1 \le i \le 3$, such that $${a_i}=3^m+2.$$
		Since $d>1000$, therefore, for $1 \le j \le 3$, $j\ne i$, ${a_j} \ne {3^m+4} $.
		
		We consider that $${{a_i}-{a_4}}={3^m-3^{m-1}+2-2^{y_4}}={2\cdot(3^{m-1}+1)-2^{y_4}},$$If there exists $1 \le j \le 3$ and $j \ne i$, such that ${a_j}=3^m+2^{y_j}$, thus ${y_j}>1$. Therefore 
		$${a_j}-{a_4}=2\cdot3^{m-1}+2^{y_j}-2^{y_4} \equiv 2~(\mod~4),$$thus $d \equiv 2~(\mod~4)$. 
		Since ${a_j}-{a_4}$ is  a multiple of 4, thus $i=2,j=1$. Thus $${a_1}=3^m+2^{y_1},{a_2}=3^m+2,{a_4}=3^{m-1}+2^{y_4}.$$Applying Proposition \ref{prop2.4}, we have $${a_3}=2^n+3^{x_3}.$$
		From above we have$${a_1}=3^m+2^{y_1},{a_2}=3^m+2,{a_3}=2^n+3^{x_3},{a_4}=3^{m-1}+2^{y_4}.$$
		
		We consider that $$2d={a_2}-{a_4}=2\cdot(3^{m-1}+1)-2^{y_4}$$ is not a multiple of 8.
		Thus $3^{m-1} \equiv 1~(\mod~8)$, therefore $${a_4}=1~(\mod~8), {a_2}=5~(\mod~8).$$
		Since $d>1000$, ${a_1} \ge {a_2}+1000$, thus ${y_1 }\ge 4$, thereby
		$${a_1} \equiv 3^m \equiv 3~(\mod~8).$$Thus $${a_3}=3^{x_3}+2^n \equiv 7~(\mod~8),$$ but $$3^{x_3} \equiv 1~(\mod~8)~or~3^{x_3} \equiv 3~(\mod~8).$$ This is a contradiction.
		
		Finally, we discuss each of the three cases of $${a_i}={3^m}+2, i=1,2,3.$$
		
		{\bf Case 1 :}\ \ When ${a_1}={3^m}+2$, then
		$${a_1}={3^m+2},{a_2}={3^{x_2}+2^n},{a_3}={3^{x_3}+2^n},{a_4}={3^{m-1}+2^{y_4}}.$$
		
		Since $$\frac{3}{4}\cdot{a_3}=\frac{3}{4}(a+4d) \le (a+3d) ={a_4},$$and ${y_4} \le n-2,$
		thereby $$\frac{3}{4}(3^{x_3}+2^n) \le {3^{m-1}+2^{y_4}} \le {3^{m-1}+2^{n-2}},$$ thus $${2^{n-1} <3^{m-1}}.$$
		
		On the other hand, we have$${2(a+3d)= 2{a_4}={a_1} +{a_7}\ge{a_1}= {a+6d} },$$ thereby $${(3^{m-1}+2^{y_4})\cdot2 \ge 3^m+2},$$
		thus $${2^{n-1} > 3^{m-1}}.$$
		This is a contradiction.
		
		{\bf Case 2 :}\ \ When ${a_2}={3^m}+2$, then
		$${a_1}={3^{x_1}+2^n},{a_2}={3^m+2},{a_3}={3^{x_3}+2^n},{a_4}={3^{m-1}+2^{y_4}}.$$
		
		Since $$\frac{3}{4}\cdot{a_3}=\frac{3}{4}(a+4d) \le (a+3d) ={a_4},$$ and ${y_4} \le n-2.$ Therefore, $$\frac{3}{4}(3^{x_3}+2^n) \le {3^{m-1}+2^{y_4}},$$ thus $${2^{n-1} <3^{m-1}}.$$
		
		On the other hand, since $${\frac{3}{5}(a+5d) < {a+3d}},$$
		thus $${3(3^m+2) <{5(3^{m-1}+2^{y_4})} \le {5\cdot3^{m-1}+5\cdot2^{n-2}}},$$ thus $${4\cdot3^{m-1} <5\cdot2^{n-2}},$$ thereby $${3^{m-1} <2^{n-1}}.$$
		This is a contradiction.
		
		{\bf Case 3 :}\ \ When ${a_3}={3^m}+2$, then
		$${a_1}={3^{x_1}+2^n},{a_2}={3^{x_2}+2^n},{a_3}={3^m+2},{a_4}={3^{m-1}+2^{y_4}}.$$
		Since$$ 4(a+3d) >3(a+4d),$$ thus $$4(3^{m-1}+2^{n-2}) >3\cdot3^m ,$$ thereby $$2^n >5\cdot3^{m-1}.$$ 
		
		On the other hand,$$5(a+3d) >3(a+5d),$$ thus $$5(3^{m-1}+2^{n-2}) >3\cdot2^n,$$ thereby $$5\cdot3^{m-1} >7\cdot2^{n-2}>2^n.$$ 
		This is a contradiction.
	\end{proof} 
	
	From the proof of the above proposition, we obtain that there is no an arithmetic progression of the form ${3^x+2^y}$ of lenth seven. The proof of theorem \ref{thm 1.1} is complete.

	\appendix
	
	\section{Appendices}
	
	\begin{figure}[h]
		\centerline{\includegraphics[width= 5 in]{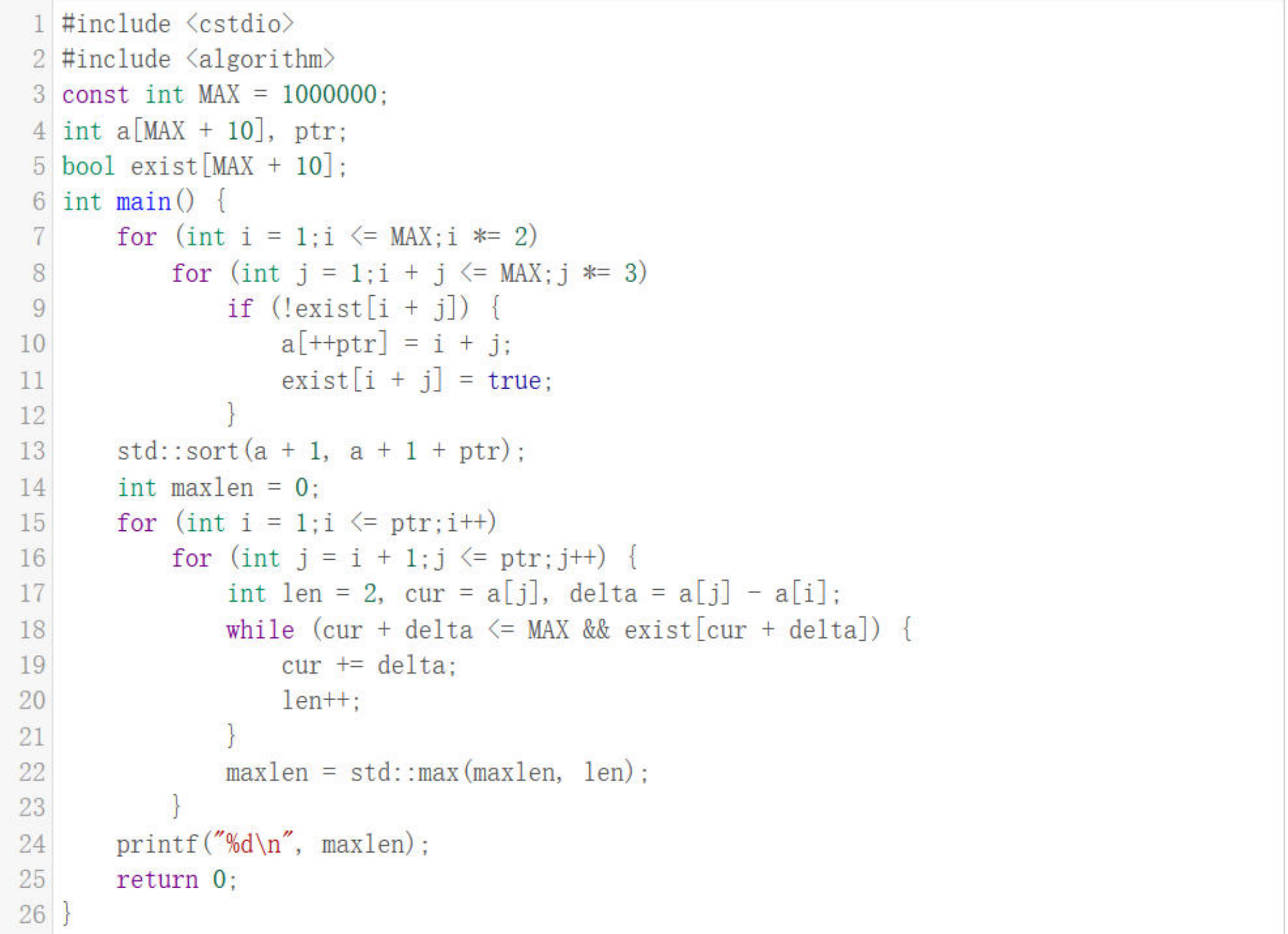}}
		\vspace*{8pt}
		\caption{program \label{fig1}}
	\end{figure}
	\FloatBarrier


\begin{thebibliography}{0}
		
		\bibitem{1}%[1]
		Iannucci, D. E. (2019). On duplicate representations as $2^ x+ 3^ y $ for nonnegative integers $ x $ and $ y$. 10.48550/arXiv. 1907.03347.
		
		
		\bibitem{L}%[2]
		Ddamulira, M. (2020). On a problem of pillai with fibonacci numbers and powers of $3$. Bolet\'{i}n de la Sociedad Matem\'{a}tica Mexicana, {\bf 26}(2), 263-277.
		
		
		\bibitem{C}%[3]
		Ddamulira, M. and Luca, F. (2020). On the problem of Pillai with k-generalized Fibonacci numbers and powers of $3$. International Journal of Number Theory, {\bf 16}(07), 1643-1666.
		
	\end{thebibliography}
\end{document}